\documentclass{amsart}
\usepackage{graphicx}
\usepackage{amscd}
\usepackage{amsmath}
\usepackage{amsfonts}
\usepackage{amssymb}
\theoremstyle{plain}
\newtheorem{theorem}{Theorem}[section]
\newtheorem{corollary}[theorem]{Corollary}
\newtheorem{lemma}[theorem]{Lemma}
\newtheorem{proposition}[theorem]{Proposition}
\theoremstyle{definition}
\newtheorem{example}[theorem]{Example}

\newtheorem{remark}[theorem]{Remark}
\theoremstyle{remark}

\begin{document}
\title{Commutative rings with two-absorbing factorization}

\author{Muzammil Mukhtar, Malik Tusif Ahmed and Tiberiu Dumitrescu}

\address{Abdus Salam School of Mathematical Sciences GCU Lahore, Pakistan}
\email{muzammilmukhtar@sms.edu.pk, muzammilmukhtar3@gmail.com (Mukhtar)}
\email{tusif.ahmed@sms.edu.pk, tusif.ahmad92@gmail.com (Ahmed)}
\address{Facultatea de Matematica si Informatica,University of Bucharest,14 A\-ca\-de\-mi\-ei Str., Bucharest, RO 010014,Romania}
\email{tiberiu@fmi.unibuc.ro, tiberiu\_dumitrescu2003@yahoo.com (Dumitrescu)}

\begin{abstract}
We use the concept of $2$-absorbing ideal introduced by Badawi to
study those commutative rings in which every   proper ideal  is a product of $2$-absorbing ideals (we call them TAF-rings). Any TAF-ring has dimension at most one and the local TAF-domains are the atomic pseudo-valuation domains.
\end{abstract}


\maketitle


\section{Introduction}

 In \cite{B}, Badawi introduced and studied the concept of $2$-absorbing ideal which is a generalization of prime ideal.  An ideal $I$ of a commutative ring $R$ is a {\em $2$-absorbing ideal} (our abbreviation {\em TA-ideal}) if whenever $a,b,c\in R$ and $abc\in I$, then $ab\in I$ or $ac\in I$ or $bc\in I$. 
 In this case, $\sqrt{I}=P$ is a prime ideal with $P^2\subseteq I$ or 
 $\sqrt{I}=P_1\cap P_2$ with $P_1$, $P_2$ incomparable prime ideals and $P_1P_2\subseteq I$, cf. \cite[Theorem 2.4]{B}. In a Pr\"ufer domain, a TA-ideal is  prime  or  a product of  two primes, cf. \cite[Theorem 3.14]{B}. 
 In \cite{AB}, Anderson and Badawi introduced and investigated the more general concept of  $n$-absorbing ideal (an ideal $I$ is {\em  $n$-absorbing} if whenever $I$ contains an $(n+1)$-factor product $P$, then $I$ contains an $n$-factor subproduct of $P$). The study of $n$-absorbing ideals continued in several other recent papers (see for instance \cite{PB}).
 
 The aim of this note  is to study the commutative rings whose ideals have a TA-factorization. 
 Let $I$ be an ideal of $R$. By a {\em TA-factorization} of $I$ we mean an expression of $I$ as a product $J_1\cdots J_n$ of proper TA-ideals. Call $R$ a {\em $2$-absorbing factorization ring (TAF-ring)} if every proper ideal has a TA-factorization. A {\em TAF-domain} is a domain which is a TAF-ring. Our paper consists of this introduction and another three sections.

 In Section 2 we present basic facts.   
 In a TAF-ring every ideal has finitely many minimal prime ideals (Proposition \ref{2P}) and  TAF-ring property is stable under factor ring (resp. fraction ring, resp. finite direct product ring) formation   (Propositions \ref{1P} and \ref{11P}). While $\mathbb{Z}[\sqrt{-7}]$ is an easy example of a non-TAF-domain, $\mathbb{Z}_8[X]/(X^2,2X)$ is a finite non-TAF-ring of smallest possible order (Proposition \ref{24P}).
 Since every prime ideal is a TA-ideal, any ZPI-ring (i.e. a ring whose proper ideals are products of prime ideals) is a TAF-ring; so a Dedekind domain is a TAF-domain. As recalled above, in a Pr\"ufer domain the TA-ideals are     products of  primes (\cite[Theorem 3.14]{B}), so the Pr\"ufer TAF-domains are the Dedekind domains. In Proposition \ref{14P}  we extend these facts to rings with zero-divisors by showing that in an arithmetical ring (i.e. a ring whose 
 ideals are locally comparable under inclusion) 
 the TA-ideals are products of  primes; so the arithmetical TAF-rings are the ZPI-rings. 
In Section 3 we show that  a TAF-ring is a finite direct product of one-dimensional domains and zero-dimensional local rings, so its  dimension is at most one (Theorems \ref{26T} and \ref{27T}).
As an application, the polynomial ring $R[X]$ is a TAF-ring if and only if    $R$ is a von Neumann regular TAF-ring if and only if    $R$ is a finite direct product of fields (Corollary \ref{19C}).
In Section 4 we study  TAF-domains.  In Theorem \ref{5T} we characterize the local  TAF-domains.   These are exactly the domains studied by Anderson and Mott in \cite[Theorem 5.1, Corollary 5.2]{AM}, that is, the (atomic) pseudo-valuation domains  (the definition of a pseudo-valuation domain (PVD) is recalled in Theorem \ref{17T}).
In Theorem \ref{3T} we show that a domain $D$ is a TAF-domain if and only if 
$D$ has finite character and it is locally a TAF-domain if and only if
$D$ is a one-dimensional ACCP-domain which has finite character and  every principal ideal generated by an atom  has a TA-factorization.
In particular, 
if $K\subseteq L$ is a field extension, then $K+XL[X]$ is a TAF-domain (Corollary \ref{22C}). While in general an overring of a TAF-domain is not TAF, this fact is true in the Noetherian case (Example \ref{25E} and Corollary \ref{20C}). In Corollary \ref{18C} we characterize the Noetherian domains with nonzero conductor as certain pull-backs of  Dedekind domains. 

Throughout this note all rings are commutative and unitary. For any undefined terminology, we refer the reader to \cite{G} or \cite{K}.

 \section{Basic facts}

We begin by examining a simple example. 

\begin{example}\label{15E}
$\mathbb{Z}[\sqrt{-7}]$ is not a TAF-domain, because the ideal  $I=(3+\sqrt{-7})$ has no TA-factorization. Indeed, $I$ is not a TA-ideal  because $2\cdot 2\cdot 4\in I$ but $8\notin I$. 
As $2^4\in I$, all proper ideals containing $I$ are contained in the maximal ideal  $(2,3+\sqrt{-7})$. So, if $I$ is a product of at least two proper (TA)-ideals, then $I\subseteq (2,3+\sqrt{-7})^{2}\subseteq (2)$, a contradiction.
\end{example}
Denote by   Min$(I)$   the set of minimal prime ideals over an ideal $I$.

\begin{proposition}\label{2P}
If $I$ is a proper ideal of a TAF-ring $R$, then Min$(I)$ is finite.
\end{proposition}
\begin{proof}
Let $I=I_1 I_2 \cdots I_n$ be a TA-factorization of $I$. Then Min$(I)\subseteq \cup _{j=1}^n Min(I_j)$ and each Min$(I_j)$ is finite by \cite[Theorem 2.3]{B}. 
\end{proof}

\begin{proposition}\label{1P}
A factor ring (resp. a fraction ring) of a TAF-ring is again a TAF-ring.
\end{proposition}
\begin{proof}
 Let $R$ be a TAF-ring and $I\subseteq J$  ideals of $R$. As $R$ is a TAF-ring, we can write $J=J_1 J_2 \cdots J_n$ with each $J_i$  a TA-ideal of $R$. Then $J/I=(J_1/I)(J_2/I) \cdots (J_n/I)$ where each $J_i/I$ is a TA-ideal of $R/I$, cf. \cite[Corollary 4.3]{AB}. So, $R/I$ is a TAF-ring. The fraction ring assertion is a consequence of the fact that a localization of a TA-factorization is still a TA-factorization, cf.  \cite[Lemma 3.13]{B}.
\end{proof}

\begin{proposition}\label{11P}
Two rings $A$ and $B$ are TAF-rings if and only if their direct product $A\times B$ is a TAF-ring 
\end{proposition}
\begin{proof}
$(\Rightarrow)$ 
An ideal $H$ of $A$ is a TA-ideal if and only if $H\times B$ is a TA-ideal of $A\times B$, because $(A\times B)/(H\times B)=A/H$.
Therefore, if $I=I_{1}I_{2}\cdots I_{m}$  is (resp. $J=J_{1}J_{2}\cdots J_{n}$) are TA-factorizations in $A$ (resp. $B$), then $I\times J =(I_{1}\times B)\cdots (I_{m}\times B)(A\times J_{1})\cdots (A\times J_{n})$ is a TA-factorization in 
$A\times B$.
$(\Leftarrow)$ follows from Proposition \ref{1P}.
\end{proof}


\begin{proposition}\label{20P}
 Let $R$ be a TAF-ring and $I$ an  ideal of $R$ with $\sqrt{I}=M\in $Max$(R)$.
 Then $I$ and $M^2$ are comparable under inclusion. 
\end{proposition}
\begin{proof}
 Assume that $M^2\not\subseteq I$. Then  $I$ is not a TA-ideal, by \cite[Theorem 2.4]{B}. As $R$ is a TAF-ring, we have 
 $I=I_{1} I_{2}\cdots I_{n}$ with each $I_{i}$ a proper TA-ideal and $n\geq 2$. From $I\subseteq I_{i}$, we get $\sqrt{I_{i}}=M$ for each $i$, so $I\subseteq M^2$.  
\end{proof}


\begin{proposition}\label{24P}
 Let $n$ be a positive integer such that $p^4\not\!| \ n$ for each prime $p$. Then every ring of order $n$ is a TAF-ring. Moreover, for each prime $p$,   $\mathbb{Z}_{p^3}[X]/(X^2,pX)$ is a non-TAF ring of order $p^4$.
\end{proposition}
\begin{proof}
 Combining Proposition \ref{11P} with the well-known fact that each finite ring  is a product of local rings of prime power order, it suffices to consider local rings of order $n=p^2$ or $n=p^3$. In the first case the possible rings  $\mathbb{Z}_{p^2}$, $\mathbb{Z}_{p}[X]/(X^2)$ and $\mathbb{F}_{p^2}$ are TAF-rings because their ideals are TA-ideals (since the square of the maximal ideal is zero).
 Assume that $(R,M)$ is a local ring of order $p^3$. Then $M^3=(0)$ and, by the preceding case, all nonzero ideals of $R$ are TA-ideals.  For the ``moreover'' part, denote by $x$ the image of $X$ in $A=\mathbb{Z}_{p^3}[X]/(X^2,pX)$. Then $xA$ and $(p,x)^2$ are not comparable  under inclusion, so $A$ is not a TAF-ring by Proposition \ref{20P}.
\end{proof}

\begin{remark}
 In particular, $T=\mathbb{Z}_8[X]/(X^2,2X)$ is a finite non-TAF-ring of smallest possible order $16$. Let $d$ be a non-square integer  congruent to $1$ mod $8$. Since  $(1+X)^2-d$ has null image in $T$, it follows that $T$ is an epimorphic image of $\mathbb{Z}[\sqrt{d}]$, so the later  is not a TAF-ring, cf. Proposition \ref{1P} (see also Corollary \ref{6C}). A similar argument works for $D=\mathbb{Z}[\sqrt[3]{d}]$, where $d$ is a noncube integer $d\equiv 1\ ($mod $9)$. Indeed, $D/9D$ is isomorphic to $R=\mathbb{Z}_9[X]/(X^3)$ which is a local ring with maximal ideal $M=(3,X)R$. An easy computer check shows that $(X^2+3)R$ and $M^2$ are not comparable under inclusion, so $R$ (and hence $D$) are not TAF-rings, by Proposition \ref{20P}.
\end{remark}

We extend to rings with zero-divisors, the description of TA-ideals of a Pr\"ufer domain given in \cite[Theorem 3.14]{B}. An \emph{arithmetical ring} is a ring whose   localizations at  maximal ideals are   chained rings.
 A ring $R$ is called a \emph{chained ring} if every two ideals of $R$ are comparable under inclusion.  
Obviously, the arithmetical (resp. chained) domains are the Pr\"ufer (resp. valuation) domains.
 
\begin{lemma} \label{10L}
Let $R$ be a chained ring and $I$  a TA-ideal of  $R$ such that $\sqrt{I}=P$ is a prime ideal. Then $I\in \{P,P^2\}$ and $I$ is $P$-primary.
\end{lemma}
\begin{proof}
By \cite[Theorem 3.6]{B}, $I$ is $P$-primary and $P^2\subseteq I$. Suppose there exists $x\in P-I$. As $R$ is a chained ring, we have $I=xJ$ for some ideal $J$ of $R$. If $J\subseteq P$, then we get $I=P^2$. Assume there exits  $y\in J-P$. As $R$ is a chained ring, we get $P\subseteq yR$, so $x=yz$ for some $z\in R$. Note that $xy=y^2z \in I$ but $y^2 \notin I$. As $I$ is a TA-ideal, we get $yz=x\in I$, a contradiction.   
\end{proof}

\begin{proposition}\label{14P}
Let $I$ be a TA-ideal of an arithmetical ring $R$. Then $I$ is a prime ideal or a product of two prime ideals. In particular, the  arithmetical  TAF-rings are exactly the ZPI rings.
\end{proposition}
\begin{proof}
By \cite[Theorem 2.4]{B}, we have one of the cases: $(1)$ $P^2\subseteq I\subseteq P$ with $P$ a prime ideal or $(2)$ $P_1P_2\subseteq I\subseteq P_{1}\cap P_{2}$ with $P_1$, $P_2$ incomparable prime ideals. 
Suppose we are in Case $1$. By Proposition \ref{1P}, $IR_M$ is a TA-ideal hence $PR_M$-primary for each maximal ideal $M$ containing $P$, cf. Lemma \ref{10L}. Thus $I$ is primary because $\sqrt{I}=P$. If there is a maximal ideal $M$ containing $P$ such that $IR_{M}=PR_{M}$, then contracting back to $R$ we  get $I=P$. Otherwise, Lemma \ref{10L} shows that $IR_{M}=P^2R_{M}$ for each maximal ideal $M$ containing $P$, so $I=P^2$.
Suppose we are in Case $2$. Then $P_{1}$ and $P_{2}$ are comaximal because $R$ is a locally  chained ring. Thus $I=P_{1}P_{2}$. 
For the ``in particular'' assertion, use the well-known fact that a ZPI ring is arithmetical.
\end{proof}

\section{TAF-rings are at most one-dimensional}

\begin{theorem}\label{26T}
Let   $R$ be a TAF-ring. Then $R_M$ is a zero-dimensional ring or a one-dimensional domain for each $M\in$ Max$(R)$. In particular, $dim(R)\leq 1$.
\end{theorem}
\begin{proof}
By Proposition \ref{1P},    we may assume that $R$ is local with maximal ideal  $M$ (and $M\neq (0)$).  It suffices to show that $M$ is the only nonzero prime ideal. Deny, so suppose that $R$ has a nonzero prime ideal  $P\neq M$. Pick an element $y\in M-P$. Shrinking $M$, we may assume that $M$ is minimal over $(P,y)$. Suppose for the moment that the following two assertions hold.

$(1)$ $M\neq M^2$.

$(2)$ Every  prime ideal $Q\subset M$ is contained in $M^2$.
\\
 By $(1)$ we can pick an element $\pi\in M-M^2$. By $(2)$ we have $\sqrt{\pi R}=M$, so  $M^2\subseteq \pi R$, hence $M^2=\pi M$, cf. Proposition \ref{20P}.
 Let $s\in P-\{0\}$ and $sR=H_1\cdots H_m$ a TA-factorization of $sR$.  As $P$ is prime, it contains some  $H_i:=H$. By $(2)$ we have $H\subseteq P\subseteq M^2=\pi M$, so,   $H=\pi J$ for some ideal $J$. From $H=\pi J\subseteq P$ and $\pi\notin P$ (because $\sqrt{\pi R}=M$), we get $J\subseteq P\subseteq M^2= \pi M$, so $J=\pi J_1$ for some ideal $J_1$, hence $H=\pi^2 J_1$. As $H$ is a TA-ideal and $\pi^2\notin H$ (because $\pi\notin P$), we get $J=\pi J_1\subseteq H$, thus $J=H$, so $H=\pi H$.
  Combining the equalities $sR=H_1\cdots H_m$ and $H=\pi H$, we get $sR=\pi sR$, so $s(1-\pi t)=0$ for some $t\in R$. Since $1-\pi t$ is a unit, we have $s=0$ which is  a contradiction. It remains to prove $(1)$ and $(2)$.

 {\em Proof of $(1)$.} Suppose that $M=M^2$. Then $M$ is the only TA-ideal with radical $M$.  As $R$ is a TAF-ring, $y\in M-P$ and $M$ is minimal over $(P,y)$,  it follows that  $(P,y^2)=M$.  We get $(P,y)=M=(P,y^2)$ which leads to a contradiction after moding out by $P$.

 {\em Proof of $(2)$.} Deny, so there exists a  prime ideal $Q\subset M$ such that $Q\not\subseteq M^2$. Pick an element  $z\in M-Q$. As $R$ is a TAF-ring
 and $Q\not\subseteq M^2$,
 it follows that $(Q,z^3)$ is a TA-ideal, hence from $z^3\in (Q,z^3)$ we get $z^2\in (Q,z^3)$ which gives a contradiction after moding out by $Q$. 
\end{proof}
 
\begin{remark}
 The proof above gives another argument for the classical fact that in a domain $D$ whose ideals are products of primes, every nonzero ideal is invertible. Indeed, such a domain is TAF, so $D$ has dimension one by Theorem \ref{26T} (the field case is trivial). It suffices to see that every maximal ideal $M$ is invertible. Given $x\in M-\{0\}$ and $xD=N_1\cdots N_n$ a prime factorization of $xD$, we get that $M$ contains some $N_i$, hence $M=N_i$, thus $M$ is invertible. 
\end{remark}

\begin{theorem}\label{27T}
Any TAF-ring is a finite direct product of one-dimensional domains and zero-dimensional local rings having nilpotent maximal ideal. In particular, a TAF-ring of dimension one having a unique height-zero prime ideal is a domain.
\end{theorem}
\begin{proof}
We adapt  the proof of \cite[Theorem 46.11]{G}. Let $R$ be a TAF-ring; so dim$(R)\leq 1$ by Theorem \ref{26T}.
Let $(0)=H_1\cdots H_m$ be a TA-factorization of $(0)$. By \cite[Theorem 2.4]{B}, every $H_i$ contains a product of (height-zero) primes, so 
$(0)=P_1^{e_1}\cdots P_n^{e_n}$ where $P_1$,...,$P_n$ are (the) distinct height-zero primes of $R$. 
By Theorem \ref{26T}, the ideals $P_1$,...,$P_n$ are pairwise comaximal, so, by Chinese Remainder Theorem, we have $R=R/P_1^{e_1}\times\cdots\times R/P_n^{e_n}$. Pick $i$ between $1$ and $n$ and set $P=P_i$ and $e=e_i$. If $P$ is maximal, then $R/P^e$ is local and zero-dimensional. Assume that $P$ is not maximal. By Theorem \ref{26T}, we have $PR_M=(0)$ for each maximal ideal $M$ containing $P$, so $P^e=P$, hence $R/P^e$ is a one-dimensional domain.
\end{proof}


We give an application.
Recall that a ring $R$ is \emph{von Neumann regular} if $R$ is zero-dimensional and reduced (equivalently,  if every localization of $R$ at a maximal ideal is a field).

\begin{corollary}\label{19C}
For a ring $R$, the following are equivalent. 

$(a)$ $R[X]$ is a TAF-ring. 

$(b)$ $R$ is a von Neumann regular TAF-ring. 

$(c)$ $R$ is a finite direct product of fields.
\end{corollary}
\begin{proof}
$(a)\Rightarrow (c)$
Since the polynomial ring formation  distributes to a finite direct product of rings, Theorem \ref{27T} enables us to assume that $R$ (and hence also $R[X]$) has a unique height-zero prime ideal. A new application of Theorem \ref{27T} gives that $R[X]$ is a one-dimensional domain, hence $R$ is a field.
$(c)\Rightarrow (a)$
If $R$ is a finite direct product of fields, then  $R[X]$ is a finite direct product of PIDs, so $R[X]$ is a TAF-ring by Proposition \ref{11P}.
$(b)\Rightarrow (c)$ By Proposition  \ref{2P}, $R$ is a semilocal von Neumann regular ring, hence a finite direct product of fields by Chinese Remainder Theorem. $(c)\Rightarrow (b)$ is clear.

\end{proof}


\section{One-dimensional TAF-domains}

%


Recall that  a local  domain $(D,M)$ is called a \textit{pseudo-valuation domain (PVD)} if it is a field or it satisfies (one of) the equivalent conditions in the next theorem.
$V$ is called the {\em associated valuation domain} of $D$. Denote $\{x\in K\mid xM\subseteq M\}$ by $(M:M)$.

\begin{theorem}{\em (Hedstorm and  Houston \cite{HH}, Anderson and Dobbs\cite{AD})}\label{17T}
For a local  domain $(D,M)$ with quotient field $K\neq D$, the following are equivalent.
 
$(a)$ $V=(M:M)$ is a valuation   domain with maximal ideal $M$.

$(b)$ If $x,y\in K-M$, then $xy\in K-M$.

$(c)$ $D$ is the pullback $\pi^{-1}(K)$ of some subfield $K$ of $V/M$, where $V$ is a valuation overring of $D$ with maximal ideal $M$ and $\pi: V\rightarrow V/M$ is the canonical map.

$(d)$ $D$ has a valuation overring $V$ with $Spec(D)=Spec(V)$.
\end{theorem}
\begin{proof}
See \cite[Theorem 2.7]{HH} and \cite[Propositions 2.5 and 2.6]{AD}.
\end{proof}

Let $D$ be a domain.
Recall that 
the \emph{Picard group Pic$(D)$} of  $D$ is the multiplicative factor group of invertible fractional ideals modulo the subgroup of nonzero principal fractional ideals.
Consequently, $Pic(D)=0$ if and only if   all invertible ideals of $D$ are principal.  

\begin{lemma}\label{4L}
 The following assertions hold.
 
 $(a)$ If $D$ is a TAF-domain with Pic$(D)=0$ and $x\in D$ is an atom, then  $xD$ is a TA-ideal. 
 
 $(b)$ Let  $(D,M)$ be a   local one-dimensional domain such that every principal ideal of $D$ has a TA-factorization. Then $D$ is atomic.
\end{lemma}
\begin{proof}
$(a)$
Let  $xD=I_1 \cdots I_n$ be a TA-factorization of $xD$.  Then each factor $I_i$ is invertible, hence principal, so $n=1$ because $x$ is an atom. 
$(b)$ 
We may assume that $M$ is not principal, otherwise $D$ is a DVR and the assertion is clear.
If $yD$ is a proper principal TA-ideal, then $M^2\subseteq yD$, so $y\notin M^2$ (note that $M^2\neq yD$, because $M$ is not principal) hence $y$ is an atom. Now pick $x\in M-\{0\}$ and let $xD=I_1\cdots I_n$ be a TA-factorization of $xD$. Then each factor $I_j$ is invertible, hence principal (since $D$ is local), thus  each $I_j$ is generated by some atom, as shown before.
Hence $D$ is atomic.
\end{proof}

For example, $\mathbb{Z}[\sqrt{-7}]$ is not a TAF-ring because it has zero Picard group (use for instance \cite[Theorem 8]{P}) and its atom $3+\sqrt{-7}$ generates a non-TA-ideal 
(see also Example \ref{15E}). 
The following result is a natural extension of \cite[Theorem 5.1, Corollary 5.2]{AM}.
Recall that a domain $D$ {\em satisfies ACCP} or is an {\em ACCP-domain} if every ascending chain of principal ideals in $D$ eventually stops. It is well known that an ACCP-domain is atomic.

\begin{theorem}\label{5T}
For a local  domain $(D,M)$ which is not a field, the following are equivalent. 

 $(a)$ $D$ is a  TAF-domain.
 
 $(b)$ $D$ is one-dimensional and every principal ideal of $D$ has a TA-factorization.

 $(c)$ $D$ is atomic, one-dimensional and every atom of $D$ generates a TA-ideal. 
 
 $(d)$ $D$ is atomic  and $M^2$ is universal (i.e. $M^2\subseteq aD$ for each atom $a\in D$).
 
 $(e)$ $D$ is an atomic PVD.
 
 $(f)$ $D$ is a PVD which satisfies ACCP.
 
 $(g)$ $(M:M)$ is a DVR with maximal ideal $M$.
\\
 If $D$ is Noetherian, we can add:
 
 $(h)$ The integral closure $D'$ of $D$ is a DVR with maximal ideal $M$.
\end{theorem}
 \begin{proof}
We may assume that $D$ is not a DVR.
$(a)\Rightarrow (b)$ follows from Theorem \ref{26T}.
$(b)\Rightarrow (c)$ follows from Lemma \ref{4L}, because a local domain has zero Picard group.
$(c)\Rightarrow (d)$ For any atom $a\in D$, we have that $aD$ is an $M$-primary TA-ideal, so  $M^2 \subseteq aD$, cf. \cite[Theorem 2.4]{B}. Thus $M^2$ is universal. 
The equivalence of $(d)$, $(e)$ and $(g)$ is given in \cite[Theorem 5.1, Corollary 5.2]{AM}.
$(f)\Rightarrow (e)$ is clear.
$(g)\Rightarrow (a)$ 
By Theorem \ref{17T}, $D$ is a PVD and $B=(M:M)$ is its associated DVR, so $D$ is one-dimensional because Spec$(D)=$ Spec$(B)$.
Pick $x\in M$ such that $M=xB$ and let  $I$ be a proper nonzero  ideal of $D$. As $B$ is a DVR, we have $IB=M^n$ for some $n\geq 1$. If $n=1$, then $M^2 = M(IB)=MI \subseteq I$, so $I$ is a TA-ideal of $D$, cf. \cite[Theorem 3.1]{B}. Next, suppose that $n\geq 2$. As $I\subseteq IB=M^n=x^{n-1}M\subseteq x^{n-1}D$, we get $I=x^{n-1}J$ for some ideal $J$ of $D$. Then $M^n=IB=M^{n-1}(JB)$, so $JB=M$ because  $M=xB$. By case $n=1$, it follows that $J$ and $xD$ are TA-ideals of $D$, so $I=(xD)^{n-1}J$ is a TA-factorization of $I$. Thus $D$ is a TAF-domain. 
$(g)\Rightarrow (f)$ is a well known consequence of the fact that $B=(M:M)$ is a DVR and $U(B)\cap D=U(D)$. When $D$ is Noetherian, the equivalence between $(g)$ and $(h)$ follows from \cite[Corollary 3.4]{HH}.
\end{proof}

Recall that  a domain  $D$ has \textit{finite character} if every $x\in D - \lbrace0\rbrace$ 
belongs to only finitely many maximal ideals of $D$. 

\begin{theorem}\label{3T}
For  a  domain $D$, the following are equivalent.

$(a)$ $D$ is a TAF-domain. 

$(b)$  $D$ has finite character and $D_{M}$ is a TAF-domain for each $M\in $ Max$(D)$.

$(c)$ $D$ has finite character and $D_{M}$ is an atomic PVD for each $M\in $ Max$(D)$.

$(d)$ $D$ has finite character and $D_{M}$ is an ACCP PVD for each $M\in $ Max$(D)$.

$(e)$ $D$ is a one-dimensional domain which has finite character and every principal ideal of $D$ has a TA-factorization.

$(f)$ $D$ is a one-dimensional ACCP-domain which has finite character and  every principal ideal generated by an atom  has a TA-factorization.
\end{theorem}
\begin{proof}
By Theorems \ref{26T} and \ref{5T}, we may assume that $D$ is one-dimensional (the field case is trivial).
$(a)\Rightarrow (b)$ follows from Propositions \ref{1P} and  \ref{2P}.
$(b)\Rightarrow (a)$ 
Every nonzero ideal $I$ of $D$ is a finite product of primary ideals. Indeed, we can check locally that $I=\prod_{i=1}^n (ID_{M_i}\cap D)$ where $V(I)=\{M_1,...,M_n\}$ (alternatively, we can invoke Jaffard's Theorem \cite[Theorem 2.1.5]{FHL}, because $D$ is h-local).
So it suffices to show that every 
nonzero $M$-primary ideal $I$ has a  TA-factorization.  As $D_{M}$ is a TAF-domain, we have $ID_M=J_1 \cdots J_n$ where each $J_i$ is a TA-ideal of $D_{M}$. Then $H_i:=J_i\cap D$ is an $M$-primary ideal of $D$ for $i=1,...,n$ and hence $H:=H_1 \cdots H_n$ is  $M$-primary. 
We have
$HD_{M}=(H_1   \cdots H_n)D_{M}=J_1 \cdots J_n=ID_{M}$, so $I=ID_M\cap D=HD_M\cap D=H=H_1  \cdots H_n$, because $I$ and $H$ are $M$-primary ideals.
Finally,  each $H_i$ is a  TA-ideal being the inverse image of the TA-ideal $I_i$ of $D_M$.
The equivalence of $(b)$, $(c)$, $(d)$ and $(e)$ follows from Theorem \ref{5T}.
$(d)\Rightarrow (f)$ follows from the well known fact that a domain is ACCP if $D$ has finite character and $D$ is locally ACCP (for instance, adapt the proof of \cite[Example 1.1, page 203]{N}).
$(f)\Rightarrow (e)$ is clear.
\end{proof}

\begin{corollary}\label{13C}
For a  Noetherian domain $D$ which is not a field, the following are equivalent.

$(a)$ $D$ is a TAF-domain.

$(b)$  $D_M$ is a TAF-domain for each $M\in$Max$(D)$.

$(c)$  $D_M$ is a PVD for each $M\in$Max$(D)$.

$(d)$   $D'_M$ is a DVR with maximal ideal $MD_M$ for each $M\in$Max$(D)$.

$(e)$  $D$ is one-dimensional and every principal ideal generated by an atom  has a TA-factorization.
\end{corollary}
\begin{proof}
Clearly a Noetherian domain satisfies ACCP (hence is atomic) and   a  one-dimensional Noetherian domain has finite character. Apply Theorems \ref{3T} and   \ref{5T}.
\end{proof}

\begin{remark}
The Noetherian {\em GMPD} domains considered in \cite{DR} are exactly the TAF-domains with all ideals two-generated, cf. \cite[Theorem 8]{DR}. So, using a Nagata-Hochster construction as done in \cite[Lemma 16 and Theorem 18]{DR}, we can  produce  a Noetherian TAF-domain $D$ with $Spec(D)$ of prescribed cardinality such that $D_M$ is not integrally closed for each $M\in$ Max$(D)$.
\end{remark}

\begin{corollary}\label{20C}
 If $D$ is a Noetherian TAF-domain, then so is every overring  of $D$.
\end{corollary}
\begin{proof}
Let $E$ be an overring of $D$ which is not a field.  By Theorem \ref{26T}, $D$ is
one-dimensional, so $E$ is Noetherian and one-dimensional, by
Krull-Akizuki Theorem. Let $Q\in $ Max$(E)$ and $P:=Q\cap D\in$Max$(D)$. By Corollary \ref{13C}, it follows that $D_P$ is a PVD, so $E_Q$ is a PVD, cf. \cite[Corollary 3.3]{HH}. A new appeal to Corollary \ref{13C} completes the proof.
\end{proof}

\begin{corollary}\label{22C}
 Let $K\subseteq L$ be a field extension. Then $K+XL[X]$ is a TAF-domain.
\end{corollary}
\begin{proof}
 Set $D=K+XL[X]$. By standard pull-back arguments, the spectrum of $D$ is $\{ M_f:=fL[X]\cap D\mid f\in L[X]$ irreducible  with $f(0)\neq 0\}\cup \{ XL[X], (0)\}$. So $D$ is one-dimensional of finite character. We have that  $D_{M_f}=L[X]_{fL[X]}$ is a DVR. Also, $D_{XL[X]}=K+XL[X]_{XL[X]}$ is a PVD, hence a TAF-domain, by Theorems \ref{17T} and  \ref{5T}. By Theorem \ref{3T}, $D$ is a TAF-domain.
\end{proof}

\begin{example}\label{25E}
Let $\mathbb{A}$ be the field of all algebraic numbers and let $\pi_1,\pi_2\in \mathbb{C}$ be algebraically independent over $\mathbb{A}$.
By Corollary \ref{22C},
$\mathbb{A}+X\mathbb{C}[X]$ is a one-dimensional integrally closed TAF-domain which is not Noetherian. Its overring $R=\mathbb{A}[\pi_1,\pi_2]+X\mathbb{C}[X]$ is not a TAF-domain, because the polynomial ring $\mathbb{A}[\pi_1,\pi_2]$
is an epimorphic image of $R$, so Proposition \ref{1P} and Theorem \ref{26T} apply. Therefore an overring of a TAF-domain is not necessarily TAF (see also Corollary  \ref{20C}).
\end{example}

\begin{corollary}\label{18C}
For a  Noetherian domain $D$ with $(D : D')\neq (0)$ the following are equivalent.

$(a)$ $D$ is a TAF-domain.

$(b)$ $D'$ is a Dedekind domain and $D_M$ is a PVD for each $M\in V(D:D')$.

$(c)$ There exist  a Dedekind overring $E$ of $D$, distinct maximal ideals $M_1$,...,$M_n$   of $E$ and   finite field extensions $K_i\subseteq E/M_i$ ($i=1,...,n$) such that $D$ is the pullback domain   $\pi^{-1}(\prod_{i=1}^n K_i)$, where $\pi:E\rightarrow \prod_{i=1}^n E/M_i$ is the canonical map.
\end{corollary}
\begin{proof}  
We may assume that $D$ is not a field.
$(a)\Rightarrow  (c)$. By Theorem \ref{26T}, $D$ has dimension one, so
$D'$ is a Dedekind domain, hence $(D : D')$ is a product $M_1\cdots M_n$ of  distinct maximal ideals. Indeed, if $M_i=M_j=P$ for some $i<j$, then $PD_P=(D_P:D'_P)=(D:D')D_P\subseteq P^2D_P$, a contradiction.   Set $N_i=M_i\cap D\in $ Max$(D)$. By Corollary \ref{13C}, $D'_{N_i}=D'_{M_i}$ is a DVR, so $N_1,...,N_n$ are distinct. We have $(D:D')=M_1\cap\cdots \cap M_n\cap D=N_1\cap\cdots \cap N_n$, so by Chinese Remainder Theorem we get $D/(D:D')=\prod_{i=1}^n D/N_i$. In other words, $D$ is the pullback $\pi^{-1}(\prod_{i=1}^n D/N_i)$ where $\pi:D'\rightarrow \prod_{i=1}^n D'/M_i$ is the canonical map.
$(c)\Rightarrow  (b)$.
Set $C=\pi^{-1}(\prod_{i=1}^n K_i)$. Since $\prod_{i=1}^n K_i\subseteq \prod_{i=1}^n E/M_i$ is finite, it follows that $C\subseteq E$ is finite, so $C$ is Noetherian by Eakin-Nagata Theorem. Clearly $C'=E$ and $(C:C')=M_1\cdots M_n$. Set $N_i=M_i\cap C\in $ Max$(C)$. 
By \cite[Lemma 1.1.6]{FHP}, we have
$C_{N_i}=\rho^{-1}(K_i)$ where $\rho:E_{M_i}\rightarrow E/M_i$ is the canonical map, so $C_{N_i}$ is a PVD, cf. Theorem \ref{17T}. Since $V(D:D')=\{N_1,...,N_n\}$, we are done.
$(b)\Rightarrow  (a)$. If $M\in$ Max$(D)-V(D:D')$, $D_M$ is a DVR. 
Apply  Corollary \ref{13C}.
\end{proof}

\begin{corollary}\label{6C}
Let $d\in\mathbb{Z}-\{0,1\}$ be a square-free integer 
$d\equiv 1$ $($mod $4)$. 
Then $\mathbb{Z}[\sqrt{d}]$ is a TAF-domain if and only if $d\equiv 5$ $($mod $8)$.
\end{corollary}
\begin{proof}
Our domain $D=\mathbb{Z}[\sqrt{d}]$ is Noetherian, $D'=\mathbb{Z}[({1+\sqrt{d}})/{2}]$ and $(D:D')=(2, 1+\sqrt{d}):=M$   is a maximal ideal of $D$. 
By Corollary \ref{18C} and Theorem \ref{5T}, $D$ is a TAF-domain if and only if $D'_M/MD_M=D'/M$ is a field.
Now $D'/M=\mathbb{Z}_2[X]/(X^2-X+(1-d)/4)$ is a field if and only if $(1+d)/4$ is an odd integer if and only if $d\equiv 5$ $($mod $8)$. 
\end{proof}

For example, in $D=\mathbb{Z}[\sqrt{-11}]$, the product $(4,3+\sqrt{-11})(5,3+\sqrt{-11})$ is a TA-factorization of  $(3+\sqrt{-11})$, because 
$D/(4,3+\sqrt{-11})=\mathbb{Z}_4$ and 
$D/(5,3+\sqrt{-11})=\mathbb{Z}_5$.

\begin{example}
 We use Corollary \ref{18C} to construct a TAF-subring of $E=\mathbb{Z}[\sqrt[4]{2}]$. It is well-known that $E$ is a Dedekind domain. Note that $3E$ is the product of the maximal ideals  $M_1=(3,y^2-y-1)$ and $M_2=(3,y^2+y-1)$, where $y=\sqrt[4]{2}$. Also note that $E/M_i=\mathbb{F}_9$ for $i=1,2$. Let $$\pi:E\rightarrow E/3E=
 \frac{\mathbb{Z}_3[T]}{(T^2-T-1)}\times \frac{\mathbb{Z}_3[U]}{(U^2+U-1)}= \mathbb{F}_9\times \mathbb{F}_9$$ be the canonical map. It can be seen that 
 $\pi(f)$ is the image in $\mathbb{F}_9\times \mathbb{F}_9$ of 
 $$ (a+c+d+(b+c+2d)T, a+c-d+(b-c+2d)U)$$
 where
 $f=a+by+cy^2+dy^3$ with $a,b,c,d\in \mathbb{Z}$.
 By Corollary \ref{18C}, $D:=\pi^{-1}(\mathbb{Z}_3\times \mathbb{Z}_3)$ is a TAF domain.  Then $f\in D$ if and only if 
 $b+c+2d$ and $b-c+2d$ are multiples of $3$ if and only if 
 $b-d$ and $c$ are multiples of $3$ if and only if 
 $f$ has the form $f=a+b(y+y^3)+e(3y^2)+k(3y^3)$ with $a,b,e,k\in \mathbb{Z}$. We get  
  $D=\mathbb{Z}[\sqrt[4]{2}+\sqrt[4]{8},3\sqrt{2},3\sqrt[4]{8}]$.
\end{example}
\
\
\

{\bf Acknowledgements.} 
The first two authors are highly grateful to ASSMS GC University Lahore, Pakistan in supporting and facilitating this research.
The third author gratefully acknowledges the warm
hospitality of Abdus Salam School of Mathematical Sciences GC University Lahore during his  visits in the period 2006-2016.

\end{document}